\documentclass[a4paper]{amsart}
\usepackage{hyperref}
\usepackage[utf8]{inputenc}
\usepackage{txfonts, amsmath,amstext,amsthm,amscd,amsopn,verbatim,amssymb, amsfonts}
\usepackage{fullpage}
\usepackage{todonotes}
\usepackage{mathtools}

\usepackage[bbgreekl]{mathbbol}
\usepackage{enumerate}

\usepackage{float}
\restylefloat{table}

\usepackage{tikz}
\usepackage{tikz-cd}
\usetikzlibrary{matrix}
\usetikzlibrary{shapes}
\usetikzlibrary{arrows}
\usetikzlibrary{calc,3d}
\usetikzlibrary{decorations,decorations.pathmorphing,decorations.pathreplacing,decorations.markings}
\usetikzlibrary{through}

\tikzset{snakeit/.style={decorate, decoration={snake, amplitude=.2mm,segment length=1mm}}}

\tikzset{ext/.style={circle, draw,inner sep=1pt}, int/.style={circle,draw,fill,inner sep=2pt},nil/.style={inner sep=1pt}}
\tikzset{cy/.style={circle,draw,fill,inner sep=2pt},scy/.style={circle,draw,inner sep=2pt},scyx/.style={draw,cross out,inner sep=2pt},scyt/.style={draw,regular polygon,regular polygon sides=3,inner sep=0.95pt}}
\tikzset{exte/.style={circle, draw,inner sep=3pt},inte/.style={circle,draw,fill,inner sep=3pt}}
\tikzset{diagram/.style={matrix of math nodes, row sep=3em, column sep=2.5em, text height=1.5ex, text depth=0.25ex}}
\tikzset{diagram2/.style={matrix of math nodes, row sep=0.5em, column sep=0.5em, text height=1.5ex, text depth=0.25ex}}
\tikzset{rowcolsep/.style={column sep=.2cm, row sep=.1cm}}
\tikzset{cross/.style={cross out,draw}}
\tikzset{every loop/.style={draw}}

\tikzset{
  crossed/.style={
    decoration={markings,mark=at position .5 with {\arrow{|}}},
    postaction={decorate},
    shorten >=0.4pt}}

\tikzset{every picture/.style={baseline=-.65ex} }

\newcommand{\tp}{{
\begin{tikzpicture}[baseline=-.55ex,scale=.2, every loop/.style={}]
 \node[circle,draw,fill,inner sep=.5pt] (a) at (0,0) {};
 \draw (a) edge[loop] (a);
\end{tikzpicture}}}

\theoremstyle{plain}
  \newtheorem{thm}{Theorem}
  \newtheorem{defi}[thm]{Definition}
  \newtheorem{prop}[thm]{Proposition}
  
  \newtheorem{cor}[thm]{Corollary}
  
  \newtheorem{lemma}[thm]{Lemma}
\theoremstyle{definition}
  \newtheorem{ex}{Example}

\newcommand{\FreeLie}{\mathrm{FreeLie}}

\newcommand{\FM}{\mathsf{FM}}

\newcommand{\bpm}{\begin{pmatrix}}
\newcommand{\epm}{\end{pmatrix}}

\newcommand{\Aut}{\mathrm{Aut}}
\newcommand{\GC}{\mathsf{GC}}

\newcommand{\HGC}{\mathsf{HGC}}

\newcommand{\bbS}{\mathbb{S}}

\DeclareMathOperator{\sgn}{sgn}
\DeclareMathOperator{\gr}{gr}

\newcommand{\Q}{\mathbb{Q}}

\newcommand{\Sp}{\mathrm{Sp}}

\newcommand{\Diff}{\mathrm{Diff}}

\newcommand{\osp}{\mathfrak{osp}}

\newcommand{\ft}{\mathfrak{t}}

\newcommand{\GCex}{\GC^\ex}

\renewcommand{\ex}{{\mathrm{ex}}}

\newcommand{\GL}{\mathrm{GL}}

\DeclareMathOperator{\Exp}{Exp}
\DeclareMathOperator{\coker}{coker}

\DeclareMathOperator{\Res}{Res}
\newcommand{\SympDer}{\mathrm{Der}^\theta(\mathbb L(\Q^{2g}))}

\author{Florian Naef}
\address{School of Mathematics\\
Trinity College\\
Dublin 2, Ireland, D02 PN40}
\email{naeff@tcd.ie}

\author{Thomas Willwacher}
\address{Department of Mathematics\\ ETH Zurich\\ R\"amistrasse 101 \\ 8092 Zurich, Switzerland}
\email{thomas.willwacher@math.ethz.ch}

\begin{document}
\title{The Johnson homomorphism, embedding calculus and graph complexes}

\thanks{
  T.W. has been partially supported by the NCCR SwissMAP, funded by the Swiss National Science Foundation.
}



\begin{abstract}
We explain how the Johnson homomorphism and the Enomoto-Satoh trace, as well as higher-loop-order generalizations, can be obtained from graph complexes originating in the Goodwillie-Weiss calculus.
This paper can be seen as an addendum to our work \cite{FNW}.
It contains little new mathematical content, but is intended to give an overview of a different viewpoint on the Johnson homomorphism, for experts working mainly in the latter area.
\end{abstract}

\maketitle

\section{Introduction}
This paper is an addendum to our recent work \cite{FNW} (joint with M. Felder). Here we explain consequences of op. cit. for the Johnson homomorphism and Johnson cokernel.
In its the most basic form the Johnson homomorphism we consider is a homomorphism of Lie algebras 
\[
J_{g,1} : \ft_{(g),1} \to \SympDer
\]
where both sides are weight-graded Lie algebras with an action of the symplectic group $\Sp(2g)$.
Here $\SympDer$ is the Lie algebra of derivations of a free Lie algebra $\mathbb L(\Q^{2g})$ generated by a $2g$-dimensional vector space, respecting the natural symplectic form. The weight grading on $\SympDer$ is the one induced from the natural grading on the free Lie algebra $\mathbb L(\Q^{2g})$ obtained by putting the generators in weight 1. The Lie algebra $\ft_{(g),1}$ arises by completion relative to $\Sp(2g)$ of the mapping class group of a punctured surface \cite{Hain}. For $g>3$ it has the following presentation (see \cite[Theorem 8.2]{HainJohnson}):\footnote{We note that our object $\ft_{(g),1}$ is denoted by $\mathfrak{u}_{g,1}$ in Hain's paper. Hain's Lie algebra $\ft_{g,1}$, which is a one-dimensional central extension, does not arise in our work. Hain explains \cite[Section 7.2]{HainJohnson} that $\ft_{g,1}$ is the rationalization of the associated graded of the Torelli group with respect to the lower central series filtration as long as $g \geq 3$.} The generators are identified with the weight 1 subspace $\gr^1 \SympDer$ and concentrated in weight 1, and the relations are defined as the kernel of the natural map $\Lambda^2 \gr^1 \SympDer \to \gr^2 \SympDer$. Using this definition of $\ft_{g,1}$, the morphism $J_{g,1}$ is just the obvious map defined by the inclusion on generators, and it is also clear that $J_{g,1}$ is an isomorphism in weights 1 and 2.

Note also that the term \emph{weight} is sometimes called \emph{degree} in the literature. For this paper we reserve the term degree to refer to the cohomological degree, which is different from the weight. For example, both $\ft_{g,1}$ and $\SympDer_g$ are concentrated in cohomological degree zero.

We have two main results. First, as a consequence to statements of \cite{FNW} we can show that $J_{g,1}$ is injective for sufficiently large genera.

\begin{thm}\label{thm:inj intro}[Corollary \ref{cor:injective} below]
    The Johnson homomorphism $J_{g,1}$ is injective in weights $W$ such that $g\geq 3W\geq 3$.
\end{thm}

Theorem \ref{thm:inj intro} improves upon earlier results by Kupers and Randal-Williams, who showed that the kernel of $J_{g,1}$ consists of trivial $\Sp(2g)$-representations (stably), see \cite[Theorem B and Section 8]{KRWnew}.


The Johnson cokernel is the cokernel of the morphism $J_{g,1}$.
There has been an interest in the literature in computing $\coker J_{g,1}$, and in particular its decomposition into irreducible $\Sp(2g)$-representations \cite{Conant, ConantKassabov, KunoSato}.
This decomposition is in principle already known, since the $\Sp(2g)$-representations on $\ft_{g,1}$ and $\SympDer_g$ may be explicitly computed with the methods of \cite{GaroufalidisGetzler}.
Our graphical methods however also provide a very explicit expression for the decomposition in terms of other well-known numbers, namely the top weight equivariant Euler characteristics of moduli spaces of curves, or equivalently the Euler characteristics of the Kontsevich (hairy) graph complex.
To define them, let $[\lambda]_{S}$ be the irreducible representation of the symmetric group $S_{|\lambda|}$ corresponding to a Young diagram $\lambda$.
Let $\alpha^k_{g,\lambda}$ be the multiplicity of $[\lambda]_{S}$ in the decomposition of the top weight cohomology $\gr_{weight}^{top} H^k(\mathcal M_{g,|\lambda|})$ into irreducible representations of $S_{|\lambda|}$.
Then we define the integers
\begin{equation}\label{equ:chidef}
\chi_{g,\lambda} := \sum_k (-1)^k \alpha^k_{g,\lambda},
\end{equation}
which we slightly abusively call equivariant Euler characterstics.
A formula for $\chi_{g,\lambda}$ has been conjectured by Zagier and later shown by Songhafouo-Tsopm\'ene--Turchin \cite{ST2} and independently by Chan--Faber--Galatius--Payne \cite{CFGP}. Since the formula is somewhat bulky we will not reproduce it here, but refer for example the main Theorem of \cite{CFGP} or the very end of the appendix of \cite{BorinskyZagier} for the expression. 

Furthermore, for a Young diagram $\lambda$ with at most $2g$ rows let $[\lambda]_{\GL(2g)}$ be the corresponding irreducible representation of $\GL(2g)$, and for a Young diagram $\mu$ with at most $g$ rows let  $[\mu]_{\Sp(2g)}$ be the corresponding irreducible representation of $\Sp(2g)$. Let $N_{\mu\lambda}$ be the integers appearing in the branching rules from $\GL(2g)$ to $\Sp(2g)$, i.e., 
\[
\Res_{\GL(2g)}^{\Sp(2g)} [\lambda]_{\GL(2g)} \cong  \bigoplus_\mu N_{\mu\lambda} [\mu]_{\Sp(2g)},
\]
see \cite[Equation 25.39]{FultonHarris}.
Then we have the following result.
\begin{thm}\label{thm:intro euler}
    For any $W\in \mathbb N$ and $g$ such that $g\geq 3W\geq 6$ we have the following decomposition into irreducible $\Sp(2g)$-representations 
    \[
    \gr^W \coker J_{g,1} \cong \bigoplus_\mu M_{\mu}^W \cdot [\mu]_{\Sp(2g)}
    \]
    with the $g$-independent multiplicities
    \begin{equation}\label{equ:MmuW}
    M_{\mu}^W = 
    \sum_{1\leq h \leq \frac W2 +1}\sum_{\lambda \atop 
    |\lambda| +2h=W+2 } (-1)^{|\lambda|} N_{\mu\lambda^T} \,\, \chi_{h,\lambda}.
    \end{equation}
    Here $\chi_{h,\lambda}$ is the $\lambda$-component of the equivariant top weight Euler characteristic of the moduli space of curves $\mathcal M_{h,|\lambda|}$.
\end{thm} 

We note that we similarly have the formula 
\[
\gr^W \ft_{g,1} \cong \sum_\mu \left( \sum_{0\leq h \leq \frac W2 +1}\sum_{\lambda \atop 
|\lambda| +2h=W+2 } (-1)^{|\lambda|+1} N_{\mu\lambda^T} \,\, \chi_{h,\lambda} \right) [\mu]_{\Sp(2g)},
\]
valid as long as $g\geq 3W\geq 6$.

\subsection*{Structure of the paper}
We will introduce some preliminaries in Section \ref{sec:threelie} and in particular recall the definition of three graph complexes of \cite{FNW}. 
Our perspective on the Johnson homomorphism is then explained in Section \ref{sec:johnson}.
The main results of the introduction are shown in Sections \ref{sec:injective} and \ref{sec:euler proof}.
The concluding Section \ref{sec:embedding} sketches how the graph complexes we consider here arise in the embedding calculus applied to surfaces. It provides a motivation and broader perspective but is strictly speaking not required for the remainder of the paper.

\section{Three dg Lie algebras of graphs -- recollections from \cite{FNW}}
\label{sec:threelie}
\subsection{Definition}
The graph complexes $\GCex_{(g)}$, $\GC_{(g),1}$ and $\GC_{(g),1}^\tp$ are defined as follows, see \cite{FNW} for details.
\begin{itemize}
\item Elements of $\GC_{(g),1}^\tp$ are $\Q$-linear series of isomorphism classes of connected, at least trivalent graphs whose vertices are decorated by elements of the first homology $H_1(\Sigma_{g,1})=H_1(\Sigma_g)$ of a genus $g$ Riemann surfaces $\Sigma_g$, or respectively its one-punctured variant $\Sigma_{g,1}$. As the superscript indicates, these graphs may have tadpoles, that is edges connecting a vertex to itself.
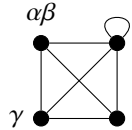
\[
\begin{tikzpicture}[yshift=-.5cm]
\node[int,label=180:{$\gamma$}] (v1) at (0,0) {};
\node[int] (v2) at (1,0) {};
\node[int,label=90:{$\alpha\beta$}] (v3) at (0,1) {};
\node[int] (v4) at (1,1) {};
\draw (v1) edge (v3) edge (v2) edge (v4) 
    (v4) edge[loop] (v4) edge (v2) edge (v3) 
    (v3) edge (v2);
\end{tikzpicture}
\quad\quad \text{with $\alpha,\beta,\gamma \in H_1(\Sigma_{g,1})$}
\]
The differential on these complexes has two terms, $\delta=\delta_{split}+\delta_{glue}$. The piece $\delta_{split}$ is defined by summing over vertices, and splitting the vertex.
  \begin{align}\label{equ:deltasplit}
    \delta_{split} \Gamma &= \sum_{v \text{ vertex} }  \pm 
    \Gamma\text{ split $v$} 
    &
    \begin{tikzpicture}[baseline=-.65ex]
    \node[int] (v) at (0,0) {};
    \draw (v) edge +(-.3,-.3)  edge +(-.3,0) edge +(-.3,.3) edge +(.3,-.3)  edge +(.3,0) edge +(.3,.3);
    \end{tikzpicture}
    &\mapsto
    \sum
    \begin{tikzpicture}[baseline=-.65ex]
    \node[int] (v) at (0,0) {};
    \node[int] (w) at (0.5,0) {};
    \draw (v) edge (w) (v) edge +(-.3,-.3)  edge +(-.3,0) edge +(-.3,.3)
     (w) edge +(.3,-.3)  edge +(.3,0) edge +(.3,.3);
    \end{tikzpicture}
  \end{align}
The piece $\delta_{glue}\Gamma$ is defined on a graph $\Gamma$ by summing over all pairs $(\alpha,\beta)$ of $H_1(\Sigma_{g,1})$-decorations in the graph $\Gamma$, replacing the pair of decorations by an edge, and multiplying the graph with the numeric prefactor $\langle \alpha,\beta\rangle$, using the canonical pairing $\langle -,-\rangle: H_1(\Sigma_{g,1})\times H_1(\Sigma_{g,1}) \to \Q$.
\item The complex 
\[
  \GC_{(g),1} = \GC_{(g),1}^\tp / I_g^{\tp}
\]
is the quotient obtained by setting all graphs with tadpoles to zero.
\item The complex $\GC_{(g)}$ is defined similarly to $\GC_{(g),1}$, except for three differences. First, one decorates vertices by $\bar H_\bullet(\Sigma_{g})$ instead of $\bar H_\bullet(\Sigma_{g,1})$. Second, the piece of the differential $\delta_{glue}$ uses the pairing 
\begin{equation}\label{equ:Hg pairing}
\langle -,-\rangle: H_\bullet(\Sigma_{g})\times H_\bullet(\Sigma_{g}) \to \Q
\end{equation}
instead.

Third, there is an additional piece of the differential $\delta_Z$ that glues a new vertex to a decoration with the top class $\omega\in H_{2}(\Sigma_{g})$. The new vertex is then decorated by the canonical diagonal element $\Delta_1\in H_1(\Sigma_{g})\otimes H_1(\Sigma_{g})$.
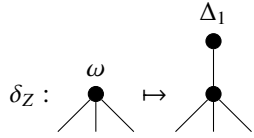
\begin{align*}
\delta_Z :
\begin{tikzpicture}
\node[int,label=90:{$\omega$}] (v) at (0,0) {};
\draw (v) edge +(-.5,-.5) edge (0,-.5) edge (.5,-.5);
\end{tikzpicture}
\mapsto 
\begin{tikzpicture}
  \node[int] (v) at (0,0) {};
  \node[int,label=90:{$\Delta_1$}] (w) at (0,.7) {};
  \draw (v) edge +(-.5,-.5) edge (0,-.5) edge (.5,-.5) (v) edge (w);
  \end{tikzpicture}
\end{align*}
\end{itemize}

We refer to \cite{FNW} for more precise definitions, including signs, prefactors and degrees. 
All three of the complexes above are in fact dg Lie algebras, with the Lie brackets defined similarly to $\delta_{glue}$ above, just operating on a pair of decorations on two distinct graphs.

Furthermore, one has a natural action of the symplectic group $\Sp(2g)$
on all three dg Lie algebras considered.
Moreover, $\GC_{(g)}$ may naturally be extended by a nilpotent, negatively graded Lie algebra $\osp^{nil}_g$ of endomorphisms of $H_\bullet(\Sigma_g)$ that respect the pairing \eqref{equ:Hg pairing}, and we will define below an extended dg Lie algebra 
\[
  \GCex_{(g)} := (\osp_g^{nil}\ltimes \GC_{(g)}, \delta).
\]

All the graph complexes above carry a natural grading  by \emph{weight}, with the weight of a graph with $e$ edges, $v$ vertices and total (homological) decoration degree $D$ defined to be the number
\[
W = 2(e-v) + D.
\]
This positive integer valued quantity is preserved by the differentials and Lie brackets. In particular our graph complexes split into a direct product of finite dimensional subcomplexes according to weight.
We shall denote the graded piece of the complexes or cohomology of given weight $W$ by the prefix $\gr^W (\cdots)$.

Finally, let us note that the complexes $\GC_{(g),1}$ and $\GC_{(g),1}^\tp$ are almost quasi-isomorphic.

\begin{lemma}[{\cite[Theorem 1(i)]{FNW}}]
\label{lem:isom}
    The natural projection to the quotient $\GC_{(g),1}^\tp\to \GC_{(g),1}$ induces an isomorphism in cohomology in all weights $W\geq 2$.
\end{lemma}

\subsection{Stable cohomology}
Let us suppose that $g\geq 6$ for simplicity and define three Lie algebras $\ft_{(g)}$, $\ft_{(g),1}$ and $\ft_{(g),1}^\tp$. Each is concentrated in cohomological degree zero, and comes with a positive weight grading and an action of the symplectic group $\Sp(2g)$.
The Lie algebras are defined by quadratic presentations as follows:
\begin{align*}
\ft_{(g)}&:=
\FreeLie(V_{(g)}) / \langle R_{(g)} \rangle \\
\ft_{(g),1}&:=\FreeLie(V_{(g),1}) / \langle R_{(g),1} \rangle  \\
\ft_{(g),1}^\tp&:=
\FreeLie(V_{(g),1}^\tp) / \langle R_{(g),1}^\tp \rangle,
\end{align*}
where the generators are concentrated in weight 1 and defined as a $\Sp(2g)$-modules as 
\begin{align*}
V_{(g)} &:= [1^3]_{\Sp(2g)}
&
V_{(g),1} &:= [1^3]_{\Sp(2g)} \oplus [1]_{\Sp(2g)}
&
V_{(g),1}^\tp := [1^3]_{\Sp(2g)}
\end{align*}
and the relations are subspaces of the exterior square of the generators such that 
\begin{align*}
\Lambda^2 V_{(g)} &= R_{(g)} \oplus [2^2]_{\Sp(2g)}
\\
\Lambda^2 V_{(g),1} &= R_{(g),1} 
\oplus [0]_{\Sp(2g)}
\oplus [1^2]_{\Sp(2g)}
\oplus [2^2]_{\Sp(2g)}
\\
\Lambda^2 V_{(g),1}^\tp &= R_{(g),1}^\tp 
\oplus [0]_{\Sp(2g)}
\oplus [1^2]_{\Sp(2g)}
\oplus [2^2]_{\Sp(2g)}
\end{align*}
There are then morphisms of doubly graded Lie algebras with $\Sp(2g)$-actions
\begin{equation}\label{equ:tgmap}
\begin{aligned}
\ft_{(g)}&\to H(\GC_{(g)}^{\ex}) \\
\ft_{(g),1}&\to H(\GC_{(g),1})  \\
\ft_{(g),1}^\tp&\to H(\GC_{(g),1}^\tp)
\end{aligned}
\end{equation}
by sending the generators to graphs with one vertex and three decorations in $H_1(\Sigma_g)$ in the natural way.
One of the main results of \cite{FNW} is then the following.
\begin{thm}[{\cite{FNW}, see also \cite{KRWnew}}]\label{thm:FNW}
The morphisms \eqref{equ:tgmap} are isomorphisms in weight $W$ as long as $g\geq 3W\geq 6$. Furthermore, the right-hand sides of \eqref{equ:tgmap} are concentrated in non-negative cohomological degrees for all $g\geq 2$ and all $W$, and in degree 0 as long as $g\geq W+2$.
\end{thm}


\section{The loop filtration on graph complexes and the Johnson homomorphism}
\label{sec:johnson}

\subsection{Definition and stable injectivity}
\label{sec:injective}
The dg Lie algebras $\GC_{(g)}^{\ex}$, $\GC_{(g),1}$ and $\GC_{(g),1}^\tp$ each carry a natural complete descending filtration by the loop order of graphs.
We shall denote by 
\begin{align*}
    \GC_{(g)}^{\ex,\geq \ell} &\subset \GC_{(g)}^{\ex}
    &
    \GC_{(g),1}^{\geq \ell}&\subset \GC_{(g),1}
    &
    \GC_{(g),1}^{\tp,\geq \ell}&\subset \GC_{(g),1}^\tp
\end{align*}
the dg Lie subalgebras spanned by graphs of loop order $\geq \ell$.\footnote{For the case $\GC_{(g)}^{\ex,\geq \ell}$ we declare that the elements of $\osp_g^{nil}\subset \GC_{(g)}^{\ex}$ have loop order zero. }
Similarly, we denote the corresponding quotient Lie algebras by 
\begin{align*}
    \GC_{(g)}^{\ex,\leq \ell} &:= \GC_{(g)}^{\ex} / \GC_{(g)}^{\ex,\geq \ell+1}
    &
    \GC_{(g),1}^{\leq \ell} &:= \GC_{(g),1}/\GC_{(g),1}^{\geq \ell+1}
    &
    \GC_{(g),1}^{\tp,\leq \ell}&\:= \GC_{(g),1}^\tp
    /
    \GC_{(g),1}^{\tp,\geq \ell+1}
    \\
    \GC_{(g)}^{\ex,= \ell} &:= \GC_{(g)}^{\ex,\geq \ell} / \GC_{(g)}^{\ex,\geq \ell+1}
    &
    \GC_{(g),1}^{= \ell} &:= \GC_{(g),1}^{\geq \ell}/\GC_{(g),1}^{\geq \ell+1}
    &
    \GC_{(g),1}^{\tp,= \ell}&\:= \GC_{(g),1}^{\tp,\geq \ell}
    /
    \GC_{(g),1}^{\tp,\geq \ell+1}
\end{align*}
These quotient Lie algebras fit naturally into towers 
\begin{equation}\label{equ:towers}
\begin{aligned}
    \GC_{(g)}^{\ex} &\to \cdots 
    \to \GC_{(g)}^{\ex,\leq \ell+1} 
    \to \GC_{(g)}^{\ex,\leq \ell} \to \cdots 
    \to \GC_{(g)}^{\ex,\leq 0}
    \\
    \GC_{(g),1} &\to \cdots \to \GC_{(g),1}^{\leq \ell+1} \to \GC_{(g),1}^{\leq \ell}\to \cdots 
    \to \GC_{(g),1}^{\leq 0}
    \\
    \GC_{(g),1}^\tp &\to \cdots
    \to \GC_{(g),1}^{\tp,\leq \ell+1}\to 
    \GC_{(g),1}^{\tp,\leq \ell}
    \to \cdots \to \GC_{(g),1}^{\tp,\leq 0}.
\end{aligned}
\end{equation}

The final objects in these towers are complexes of trees.
In particular, 
$$
\gr^W H(\GC_{(g),1}^{\tp,\leq 0}) = \gr^W  H(\GC_{(g),1}^{\leq 0})=\gr^W  \SympDer
$$
is concentrated in degree zero and identified with the Lie algebra of symplectic derivations for all $W\geq 1$.

\begin{defi}
    We call the morphisms of graded Lie algebras
    \begin{align*}
        \begin{aligned}
        J_g : \ft_{(g)}&\to H(\GC_{(g)}^{\ex,\leq 0}) \\
        J_{g,1} : \ft_{(g),1}&\to H(\GC_{(g),1}^{\leq 0})  \\
        J_{g,1}^\tp: \ft_{(g),1}^\tp&\to H(\GC_{(g),1}^{\tp,\leq 0})
\end{aligned}
    \end{align*}
    obtained as the composition of \eqref{equ:tgmap} and \eqref{equ:towers} the Johnson-homomorphisms.
\end{defi}

\begin{prop}
    The morphisms $H^0(\GC_{(g),1}^{\leq \ell+1}) \to \GC_{(g),1}^{\leq \ell}$ and $H^0(\GC_{(g),1}^{\tp,\leq \ell+1})\to 
    H^0(\GC_{(g),1}^{\tp,\leq \ell})$ induced by \eqref{equ:towers} are monomorphisms.
\end{prop}
\begin{proof}
We have a short exact sequence 
\[
0\to \GC_{(g),1}^{= \ell+1} \to \GC_{(g),1}^{\leq \ell+1}
\to \GC_{(g),1}^{\leq \ell} \to 0.
\]
Degree bounds for the cohomology $H(\GC_{(g),1}^{=\ell+1})$ have been derived in \cite[Section 6]{FNW}, with the result that 
\[
H^k(\GC_{(g),1}^{= \ell+1}) = 0 \quad \text{for $k< 0$}.
\]
Furthermore, it is has also been remarked in loc. cit. (see the proof of \cite[Proposition 24]{FNW}) that the degree bound can only be satisfied by tree graphs. Hence for $\ell \geq 0$ we have in fact the improved statement
\begin{equation}\label{equ:eq ell van}
H^k(\GC_{(g),1}^{= \ell+1}) = 0 \quad \text{for $k\leq 0$}.
\end{equation}
The proposition then follows immediately from the long exact sequence in cohomology.
For $\GC_{(g),1}^{\tp}$ the argument is analogous.
\end{proof}

\begin{cor}\label{cor:injective}
    The Johnson homomorphisms $J_{g,1}$ and $J_{g,1}^\tp$ are injective in weight $W$ as long as $g\geq 3W\geq 6$.
\end{cor}

In the following, we shall focus our attention to the Johnson homomorphism $J_{g,1}$. By Lemma \ref{lem:isom} the case of $J_{g,1}^\tp$ is treated almost identically. The case of $J_g$ is left for future work. (We expect very similar results.)

\subsection{The spectral sequence and higher Enomoto-Satoh traces}
Next we study the image and the cokernel of the Johnson homomorphism.
This can be studied nicely using the spectral sequence associated to the loop order filtration.
Then the degree-zero-part of the $\ell$-th page of the spectral sequence is $H^0(\GC_{(g),1}^{\leq \ell})$.
The differential on the $\ell$-th page is a map 
\[
ES_{g,1,\ell} : 
H^0(\GC_{(g),1}^{\leq \ell})
\to 
X_{g,1,\ell+1} \subset H^1(\GC_{(g),1}^{= \ell}),
\]
where we denote by $X_{g,1,\ell+1}$ the degree 1 part of the $\ell$-th page of the spectral sequence. From the vanishing \eqref{equ:eq ell van} we in fact see that this is a subspace of the more easily computable space $H^1(\GC_{(g),1}^{= \ell})$, and hence it is convenient to directly define the maps 
\[
ES_{g,1,\ell} : 
H^0(\GC_{(g),1}^{\leq \ell})
\to 
 H^1(\GC_{(g),1}^{= \ell})
\]
to take values in that space.

\begin{defi}
    We call the morphism $ES_{g,1,\ell}$ the $\ell$-th Enomoto-Satoh trace.
\end{defi}

We have that 
\[
H^0(\GC_{(g),1}^{\leq \ell}) = \ker ES_{g,1,\ell}, 
\]
and by convergence of the spectral sequence 
\[
H^0(\GC_{(g),1})
=
\cap_{\ell \geq 0}
\ker ES_{g,1,\ell}.
\]

So in the cases that the morphism \eqref{equ:tgmap} is a quasi-isomorphism in degree zero (including in particular $g\geq 3W$), we have that the image of the Johnson homomorphism is precisely the joint kernel of all the higher Enomoto-Satoh traces. 
Or put differently, the Johnson cokernel may be identified as a vector space with the union of the images of the higher Enomoto-Satoh traces.

Let us more explicitly work out the definition.
Suppose that we start with some linear combination of trees $\Gamma\in \GC_{(g),1}^{\leq 0}$ that is closed, i.e., $\delta_{split}\Gamma=0$. It represents a cohomology class (i.e., an element of the symplectic derivation Lie algebra) $[\Gamma]$, and we desire to compute its Enomoto-Satoh trace(s).

The first (and classical) Enomoto-Satoh trace is just the application of the part of the differential $\delta_{glue}$ that raises the loop order by one:
\[
ES_{g,1,0} (\Gamma) =\delta_{glue}\Gamma.
\]
More precisely, $\delta_{glue}\Gamma$ is a $\delta_{split}$-cocycle in the $1$-loop graph complex $\GC_{(g),1}^{=1}$, and the cohomology class represented by this element is the image of $[\Gamma]$ under $ES_{g,1,0}$. We also note that $H(\GC_{(g),1}^{=1},\delta_{split})$ is easily evaluated and is the same as the dihedral words in $H_1(\Sigma_g)$ (with certain signs).

Next suppose that our cocycle $[\Gamma]$ is in fact sent to zero under $ES_{g,1,0}$. This means that there is some linear combination $\Gamma_1\in \GC_{(g),1}^{=1}$ of graphs of loop order 1 such that 
\[
\delta_{glue}\Gamma + \delta_{split} \Gamma_1 =0.
\]
Then the second Enomoto-Satoh trace is defined as 
\[
ES_{g,1,1}(\Gamma) = [\delta_{glue} \Gamma_1]
\in 
H^1(\GC_{(g),1}^{=2}).
\]
Continuing further, if $ES_{g,1,1}(\Gamma)$ were to vanish, i.e., there is some linear combination $\Gamma_2$ of graphs of loop order 2 such that 
\[
\delta_{glue}\Gamma_1 + \delta_{split} \Gamma_2 =0,
\]
then the next higher Enomoto-Satoh trace is well-defined and represented by 
\[
ES_{g,1,2}(\Gamma) = [\delta_{glue} \Gamma_2] \in H^1(\GC_{(g),1}^{=3}),
\]
and so forth.

\subsection{The Johnson cokernel and proof of Theorem \ref{thm:intro euler}}
\label{sec:euler proof}
The Johnson cokernel is the cokernel of the Johnson homomorphism.
From Theorem \ref{thm:FNW} we directly obtain a graph complex model.
\begin{cor}\label{cor:coker}
    For $g\geq 3W\geq 6$ there are isomorphisms 
    \begin{equation}\label{equ:coker GC}
        \begin{aligned}
        \gr^W \coker J_{g,1} &\cong \gr^W H^1(\GC_{(g),1}^{\geq 1})  \\
        \gr^W \coker J_{g,1}^\tp &\cong \gr^W H^1(\GC_{(g),1}^{\tp,\geq 1})
        \end{aligned}.
    \end{equation}
    Furthermore, in this case the cohomology of 
    $\gr^W \GC_{(g),1}^{\geq 1}$ and $\gr^W \GC_{(g),1}^{\tp,\geq 1}$ is concentrated in degree one.
\end{cor}

With this result we may finally prove Theorem \ref{thm:intro euler} from the introduction.

\begin{proof}[Proof of Theorem \ref{thm:intro euler}]
Fix $W\geq 2$ and $g\geq 3W$, and for a Young diagram $\mu$ define the number $M_\mu^W$ as the multiplicity of the representation $[\mu]_{\Sp(2g)}$ inside $\gr^W \coker J_{g,1}$. We then desire to show that \eqref{equ:MmuW} holds.

Let $\beta_\mu^{W,k,\ell}$ be the multiplicity of $[\mu]_{\Sp(2g)}$ in the subspace of cohomological degree $k$ 
in $\gr^W \GC_{(g),1}^{=\ell}$. Then by Corollary \ref{cor:coker}, in particular since the cohomology is concentrated in degree 1, we have
\[
M_\mu^W = -\sum_{1\leq \ell \leq \frac W2+1} \underbrace{\sum_k (-1)^k \beta_\mu^{W,k,\ell}}_{=:\xi_{\mu}^{W,\ell}},
\]
where $\xi_{\mu}^{W,\ell}$ can be considered as the $\mu$-equivariant Euler characteristic of $\gr^W \GC_{(g),1}^{=\ell}$.

The latter is in fact known quite explicitly, using that it can be expressed through the Euler characteristic of the Kontsevich commutative hairy graph complex, which in turn computes the top weight part of the cohomology of moduli spaces. Let us explain this link, referring also to \cite[Section 6]{FNW}.
The hairy graph complexes $\HGC_{2}^{=\ell}(r)$ are complexes of graphs defined similarly to $(\GC_{(g),1}^{=\ell},\delta_{split})$, except that in place of the decorations by $H_1(\Sigma_g)$ one has $r$ hairs that are numbered $1,\dots,r$, and the degree is shifted, so that
\[
\gr^{2\ell+r-2} \GC_{(g),1}^{=\ell} [1]
\cong 
\HGC_{2}^{=\ell}(r)
\otimes_{S_r} (H_1(\Sigma_g))^{\otimes r}.
\]
Now by the result of \cite{CGP2} the same complexes $\HGC_{2}^{=\ell}(r)$ compute the compactly supported weight zero cohomology of the moduli spaces of curves $\mathcal M_{\ell,r}$, so that we have, as $S_r$-module
\[
\gr^0 H^k_c(\mathcal M_{\ell,r}) \cong 
H^{k-2\ell-r}( \HGC_{2}^{=\ell}(r) ) \otimes \sgn_{r}.
\]
The left-hand side is also equal to (the dual space of) $\gr^{top} H^{6\ell -6+2r-k}(\mathcal M_{\ell,r})$. Hence for a Young diagram $\lambda$ with $|\lambda|=r$ the $\lambda$-equivariant Euler characteristic\footnote{By this we mean the alternating sum of the multiplicities of $[\lambda]_{S_r}$ in the various degree components.} of $\gr^W\HGC_{2}^{=\ell}(r)$ is identified with 
\[
-(-1)^r \chi_{\ell,\lambda^T},
\]
with $\chi_{\ell,\sigma}$ as in \eqref{equ:chidef}.

By Schur-Weyl duality we also know that the $S_r\times \GL(2g)$-representation $(H_1(\Sigma_g))^{\otimes r}$ is expressed through irreducibles as 
\[
(H_1(\Sigma_g))^{\otimes r} \cong 
\bigoplus_{\lambda\atop |\lambda|=r} [\lambda]_{\GL(2g)} \otimes [\lambda]_{S_r}, 
\]
where the sum is over all Young diagrams with $r$ boxes. (In the unstable situation we would additionally restrict to Young diagrams with $\leq 2g$ rows, but here this condition can be omitted.)
Putting everything together we find that 
\[
M_\mu^W = -\sum_{1\leq \ell \leq \frac W2+1}\sum_k (-1)^k \beta_\mu^{W,k,\ell}
=
-\sum_{1\leq \ell \leq \frac W2+1}
\sum_{\lambda \atop |\lambda|+2(\ell-1)=W}
N_{\mu\lambda}
(-(-1)^{|\lambda|} \chi_{\ell,\lambda^T})
=
\sum_{1\leq \ell \leq \frac W2+1}
\sum_{\lambda \atop |\lambda|+2(\ell-1)=W}
(-1)^{|\lambda|} N_{\mu\lambda^T} \chi_{\ell,\lambda}
\]
as desired.
\end{proof}

The Jonhson cokernel in weights $\leq 6$ will be computed more explicitly as an illustration in the following next section. Using the above formula we moreover obtain 
\begin{align*}
\gr^7 \coker J_{g,1} &=
    2[52]_\Sp \oplus 6[41]_\Sp \oplus 11[3]_\Sp \oplus [7]_\Sp \oplus [43]_\Sp \oplus 6[32]_\Sp \oplus 13[21]_\Sp \oplus 4[1]_\Sp \oplus 3[421]_\Sp \oplus 11[31^2]_\Sp
    \\&\quad\quad\quad \oplus [41^3]_\Sp \oplus 3[32^2]_\Sp \oplus 6[2^21]_\Sp \oplus 2[321^2]_\Sp \oplus 5[21^3]_\Sp \oplus 3[1^3]_\Sp \oplus 3[31^4]_\Sp \oplus [2^31]_\Sp 
    \\
  \gr^8 \coker J_{g,1} &=  2[53]_\Sp \oplus 17[42]_\Sp \oplus 38[31]_\Sp \oplus 27[2]_\Sp \oplus 3[61^2]_\Sp \oplus 2[6]_\Sp \oplus 8[51]_\Sp \oplus 4[521]_\Sp \oplus 15[41^2]_\Sp
  \\&\quad\quad\quad 
  \oplus 7[4]_\Sp \oplus 2[51^3]_\Sp \oplus 5[431]_\Sp \oplus 6[3^2]_\Sp \oplus 27[321]_\Sp \oplus 17[2^2]_\Sp \oplus 40[21^2]_\Sp \oplus 17[1^2]_\Sp 
  \\&\quad\quad\quad 
  \oplus [42^2]_\Sp \oplus 7[421^2]_\Sp \oplus 18[31^3]_\Sp \oplus [41^4]_\Sp \oplus 4[3^22]_\Sp \oplus 10[2^3]_\Sp \oplus 3[3^21^2]_\Sp \oplus 13[2^21^2]_\Sp
  \\&\quad\quad\quad 
  \oplus 9[1^4]_\Sp \oplus 4[32^21]_\Sp \oplus 4[321^3]_\Sp \oplus 9[21^4]_\Sp \oplus [31^5]_\Sp \oplus 3[2^31^2]_\Sp \oplus [2^21^4]_\Sp \oplus [1^6]_\Sp \oplus [21^6]_\Sp \oplus [0]_\Sp 
\end{align*}

The degree-7 part (and the degree-8 part up to one summand) was previously computed by Morita, Sakasai and Suzuki and presented at the Johnson Homomorphism Workshop (2019) \cite{MSS2019}.

The numbers $\chi_{g,\lambda}$ can be computed numerically on a laptop at least up to weights $>20$ from the formulas of \cite{CFGP,BorinskyZagier, ST2}, so the Johnson cokernel can also be computed significantly further by our methods, but the results become too bulky to reproduce here.

\subsection{Explicit example: The Johnson cokernel and ES traces in weights $\leq 6$}
To exemplify our constructions more explicitly, let us compute the Johnson cokernel up to weight 6 (recovering the computations in \cite{MSS15}). We will use the loop order spectral sequence on $\GC_{(g),1}^{\geq 1}$ to compute $H^1 (\GC_{(g),1}^{\geq 1})$. 
By definition of the weight it is clear that only loop orders $\ell\leq \frac W2+1$ can contribute.
The 1-loop cohomology $H(\GC_{(g),1}^{=1})$ is concentrated in degree 1 and identified with dihedral words in $H_1(\Sigma_g)$.
Explicitly, let us record here the decomposition into irreps of the symplectic group:
\begin{align*}
    \gr^3 H^1(\GC_{(g),1}^{=1}) &= [3]_{\Sp} 
    \\
    \gr^4 H^1(\GC_{(g),1}^{=1}) 
    &= [21^2]_{\Sp} \oplus [2]_{\Sp} \oplus [1^2]_{\Sp}
    \\
    \gr^5 H^1(\GC_{(g),1}^{=1}) &= 
    [5]_{\Sp} \oplus [32]_{\Sp}
    \oplus [2^21]_{\Sp}\oplus 2[21]_{\Sp}
    \oplus [1^5]_{\Sp}
    \oplus 2[1^3]_{\Sp}\oplus 3[1]_{\Sp}
    \\
    \gr^6 H^1(\GC_{(g),1}^{=1}) &= 
      2[41^2]_\Sp \oplus [3^2]_\Sp \oplus [321]_\Sp \oplus [31^3]_\Sp \oplus [2^21^2]_\Sp \oplus 2[4]_\Sp \oplus 4[31]_\Sp \oplus 3[2^2]_\Sp \oplus 3[21^2]_\Sp \oplus [1^4]_\Sp 
      \\&\quad \quad \quad
      \oplus 2[2]_\Sp \oplus 4[1^2]_\Sp \oplus 2[0]_\Sp 
    %
\end{align*}

The 2-loop graph cohomology $H^\bullet(\GC_{(g),1}^{=2})$ is currently not fully known, although partial computations exist in the literature \cite{CCTW, BCGH, BrunWillwacher}. For weights $\leq 6$ it can for example be read off from the tables of \cite[Figure 7]{ BrunWillwacher}: 
\begin{align*}
    \gr^4 H^2(\GC_{(g),1}^{=2}) &= [1^2]_{\Sp} \oplus [0]_{\Sp} 
    \\
    \gr^6 H^1(\GC_{(g),1}^{=2}) &= [1^4]_{\Sp} \oplus [1^2]_{\Sp} \oplus [0]_{\Sp}
    \\
    \gr^6 H^2(\GC_{(g),1}^{=2}) &= [31]_{\Sp} \oplus [2]_{\Sp}
\end{align*}
with all other cohomology classes in weight $\leq 6$ vanishing. For loop order 3 we get from \cite{BrunWillwacher}
\begin{align*}
    \gr^4 H^3(\GC_{(g),1}^{=3}) &= [0]_{\Sp}
    &
    \gr^5 H^2(\GC_{(g),1}^{=3}) &= [1]_{\Sp} 
\end{align*}
while for loop order 4 the cohomology vanishes identically in weights $\leq 8$.

Since the cohomology must be concentrated in degree 1, the above data suffices to infer the higher differentials in the spectral sequence, i.e., the (higher) Enomoto-Satoh traces. This yields the following result:
\begin{align*}
    \gr^3 \coker J_{g,1} &= [3]_\Sp \\
    \gr^4 \coker J_{g,1} &= [21^2]_\Sp \oplus [2]_\Sp \\
    \gr^5 \coker J_{g,1} &= [5]_{\Sp} \oplus [32]_{\Sp}
    \oplus [2^21]_{\Sp}\oplus 2[21]_{\Sp}
    \oplus [1^5]_{\Sp}
    \oplus 2[1^3]_{\Sp}\oplus 2[1]_{\Sp}
    \\
    \gr^6 \coker J_{g,1} &= 
    2[41^2]_\Sp
    \oplus [3^2]_\Sp
    \oplus [321]_\Sp
    \oplus [31^3]_\Sp
    \oplus [2^21^2]_\Sp
    \oplus 2[4]_\Sp
    \oplus 3[31]_\Sp
    \oplus 3[2^2]_\Sp
    \oplus 3[21^2]_\Sp 
    \oplus [1^4]_\Sp
    \\&\quad\quad\quad
    \oplus [2]_\Sp
    \oplus 4 [1^2]_\Sp
    \oplus 2[0]_\Sp 
    \oplus \underbrace{[1^4]_\Sp\oplus [1^2]_\Sp\oplus [0]_\Sp}_{\text{from $\ell=2$ loops}}
\end{align*}

Note that for weights $\leq 5$ the only degree 1 cohomology is from the part of loop order 1. Hence there the Johnson cokernel is fully detected by the classical Enomoto-Satoh trace $ES_{g,1,1}$.
The first non-trivial classes in degree 1 of higher loop order ($\ell=2$) occur in weight 6 and are represented by graphs of the form 
\[
\sum_{\sigma\in S_4}(-1)^\sigma
\begin{tikzpicture}
    \node[int,label=90:{$\alpha_{\sigma(1)}$}] (v1) at (0,1) {};
    \node[int,label=0:{$\alpha_{\sigma(2)}$}] (v2) at (1.5,0) {};
    \node[int,label=0:{$\alpha_{\sigma(3)}$}] (v3) at (0,0) {};
    \node[int,label=180:{$\alpha_{\sigma(4)}$}] (v4) at (-1.5,0) {};
    \node[int] (v5) at (0,-1) {};
    \draw (v1) edge (v2) edge (v3) edge (v4) 
    (v5) edge (v2) edge (v3) edge (v4) ;
\end{tikzpicture}\, .
\]
The whole subspace spanned by such graphs is in the image of the second Enomoto-Satoh trace $ES_{g,1,1}$.


\section{Concluding remarks: Diffeomorphism groups and the embedding calculus}
\label{sec:embedding}
Our graphical dg Lie algebras $\GC_{(g)}^{ex}$, $\GC_{(g),1}$ and $\GC_{(g),1}^\tp$ arise naturally from 2-dimensional embedding calculus, as we shall briefly explain here.
Let again $\Sigma_g$ be a genus $g$ Riemann surface, and let $\Sigma_{g,1}$ be a genus $g$ Riemann surface with one boundary component. Our goal is to understand the group $\Diff(\Sigma_g)$ of oriented diffeomorphisms of $\Sigma_g$ and the group $\Diff_{\partial}(\Sigma_{g,1})$ of oriented diffeomorphisms of $\Sigma_{g,1}$ that fix the boundary pointwise. Since $\Sigma_{g,1}$ is furthermore parallelizable, we may also consider framed diffeomorphisms $\Diff_{\partial}^{fr}(\Sigma_{g,1})$.
Furthermore, there exists another related version of the diffeomorphism group of $\Sigma_{g,1}$ that is obtained by dropping the boundary condition, and moreover one can also consider self-embeddings instead of diffeomorphisms. While much (or all) of the theory can be adapted to such situations, we focus on the three base-cases above for simplicity.

Any diffeomorphism of $\Sigma_g$ acts naturally on the Fulton-MacPherson-compactified configuration spaces of framed points $\FM_{\Sigma_g}^{fr}$ of $\Sigma_g$. This action furthermore respects the right action of the framed Fulton-MacPherson operad $\FM_2^{fr}$. In the case of $\Sigma_{g,1}$ one also preserves in additions the action of the version $\FM_{2,1}^{fr}$ of $\FM_{2}^{fr}$ with one marked point, and in the case of framed diffeomorphisms one may drop the framings on the operads. Hence one obtains morphisms of topological monoids
\begin{equation}\label{equ:emb calc}
\begin{aligned}
\Diff(\Sigma_g) &\to \Aut_{\FM_2^{fr}}^h(\FM_{\Sigma_g}^{fr})
\\
\Diff_\partial(\Sigma_{g,1}) &\to \Aut_{\FM_2^{fr}-\FM_{2,1}^{fr} }^h(\FM_{\Sigma_{g,1}}^{fr})
\\
\Diff(\Sigma_{g,1}) &\to \Aut_{\FM_2- \FM_{2,1}}^h(\FM_{\Sigma_{g,1}}).
\end{aligned}
\end{equation}
The main result of \cite{KrannichKupers} (essentially) states that the morphisms above are weak equivalences.
The automorphism groups of configuration spaces (operadic modules) on the right-hand side of \eqref{equ:emb calc} are still difficult to compute in practice. However, one can compute the automorphisms groups of their rationalizations. There are morphisms of topological groups 
\begin{equation}\label{equ:emb calc Q}
\begin{aligned}
\Aut_{\FM_2^{fr, \Q}}^h(\FM_{\Sigma_g}^{fr})
&\to 
\Aut_{\FM_2^{fr}}^h(\FM_{\Sigma_g}^{fr, \Q})
\\
\Aut_{\FM_2^{fr}-\FM_{2,1}^{fr} }^h(\FM_{\Sigma_{g,1}}^{fr})
&\to
\Aut_{\FM_2^{fr,\Q}-\FM_{2,1}^{fr,\Q} }^h(\FM_{\Sigma_{g,1}}^{fr,\Q})
\\
\Aut_{\FM_2, \FM_{2,1}}^h(\FM_{\Sigma_{g,1}})
&\to
\Aut_{\FM_2^\Q, \FM_{2,1}^\Q}^h(\FM_{\Sigma_{g,1}}^\Q).
\end{aligned}
\end{equation}

The right-hand sides of \eqref{equ:emb calc Q} in turn can be computed, see \cite{Matteo, SchwarzWillwacher}, and can be expressed through the exponential groups of our graph complexes.
\begin{equation}\label{equ:emb calc GC}
\begin{aligned}
\Aut_{\FM_2^{fr}}^h(\FM_{\Sigma_g}^{fr, \Q})
&\simeq
\Sp(2g) \ltimes \Exp(\GCex_g)
\\
\Aut_{\FM_2^{fr,\Q}-\FM_{2,1}^{fr,\Q} }^h(\FM_{\Sigma_{g,1}}^{fr,\Q})
&\simeq
\Sp(2g) \ltimes \Exp(\GC_{g,1})
\\
\Aut_{\FM_2^\Q, \FM_{2,1}^\Q}^h(\FM_{\Sigma_{g,1}}^\Q)
&\simeq
\Sp(2g) \ltimes \Exp(\GC_{g,1}^\tp).
\end{aligned}
\end{equation}

These graph complexes can hence be seen as rational-embedding-calculus-avatars of the (respective version of the) Torelli group. In the unstable situation it is at present not yet clear how to make this relation more precise, and whether for example the cohomology of the moduli spaces of curves may be computed from knowledge of these Lie algebras.

\end{document}